\documentclass{conm-p-l}
\usepackage{amsfonts, amssymb, amsmath, amsthm, euscript}

\swapnumbers
\newtheorem{theorem}{Theorem}[section]

\newtheorem{lem}[theorem]{Lemma}
\newtheorem{prop}[theorem]{Proposition}
\newtheorem{cor}[theorem]{Corollary}
\newtheorem*{thmA*}{Theorem A}
\newtheorem*{thmB*}{Theorem B}

\theoremstyle{definition}

\newtheorem{exam}[theorem]{Example}

\newtheorem{notation}[theorem]{Notation and Conventions}
\newtheorem{organization}[theorem]{Organization}

\newtheorem*{qn*}{Question}

\theoremstyle{remark}
\newtheorem{rem}[theorem]{Remark}

\numberwithin{equation}{section}

\newcommand{\FF}{\mathbb{F}}

\newcommand{\QQ}{\mathbb{Q}}

\newcommand{\ZZ}{\mathbb{Z}}

\newcommand \calF {{\mathcal{F}}}
\newcommand \calO {{\mathcal{O}}}
\newcommand \calS {{\mathcal{S}}}

\newcommand \calW {{\mathcal{W}}}

\newcommand\gfrak{\mathfrak{g}}
\newcommand\hfrak{\mathfrak{h}}

\newcommand\bfq{\boldsymbol{q}}

\DeclareMathOperator \Fract { {\mathrm{Fract}} }
\DeclareMathOperator{\ch}{char}
\DeclareMathOperator{\Id}{Id}
\DeclareMathOperator{\card}{card}

\newcommand \Zpos {\ZZ_{> 0}}

\newcommand{\kx}{k^\times}
\DeclareMathOperator{\op}{op}
\newcommand\Obfq{\calO_{\bfq}}

\begin{document}

\title[Nonaffine skew fields]{Skew fields not
finitely generated as algebras}

\author{K. R. Goodearl and E. S. Letzter}

\address{Department of Mathematics\\ University of California \\Santa
  Barbara, CA 93106}

\email{goodearl@math.ucsb.edu}

\address{Department of Mathematics\\
        Temple University \\ Philadelphia, PA 19122}
      
      \email{letzter@temple.edu}

\begin{abstract} An associative division algebra $D$ is said to be \emph{affine} over a central subfield $k$ if $D$ is finitely generated as a $k$-algebra. In 1956 Amitsur famously proved that, when $k$ is uncountable, $D$ cannot be $k$-affine unless $D$ is algebraic over $k$. In this paper we consider affineness -- and nonaffineness -- for certain naturally occurring classes of division algebras over arbitrary fields. The primary applications are to division algebras of fractions of suitably conditioned iterated skew polynomial rings over $k$, including many examples naturally arising in Lie theoretic and quantum group settings.  Many transcendental division algebras are thus verified to be nonaffine over $k$.  Division algebras of fractions of Weyl algebras and quantum affine spaces are determined to be affine over their centers exactly when they are finite dimensional over their centers.
\end{abstract}



\dedicatory{
The second author dedicates this paper to the memories of his daughter Eliana \\ Risman Letzter (1995--2024) and sister Dr.~Gail Rebecca  Letzter  (1960--2024).}
 

\maketitle 


\section{Introduction}

For the general class of (associative) division algebras $D$ over a central subfield $k$, the question of whether or not $D$ can be both infinite $k$-dimensional and $k$-aff\-ine (i.e., finitely generated as a $k$-algebra) remains one of the oldest unsolved problems in noncommutative algebra. The seemingly more specialized (but still open) question of whether or not $D$ can be infinite $k$-dimensional, $k$-affine, and also algebraic over $k$ is the well known Kurosh Problem for Division Algebras, first appearing in 1941 \cite{Kur}. (The reader is referred, e.g., to \cite{Smo} and \cite{Zel} for brief historical overviews.) 

Notice that there is no loss of generality in this discussion if we assume further that $D$ is finitely generated as a skew field over $k$; that is, there is a finite subset $S \subseteq D$ such that $D$ is the smallest division $k$-subalgebra of itself containing $S$. (We will use the terminology \emph{finitely generated skew field} to help distinguish our use of this concept from that of affine algebra.) 

The most important fundamental result on affineness of division algebras is Amitsur's famous proof \cite{Ami}, for uncountable $k$,  that if $D$ is not algebraic over $k$ then $D$ cannot be $k$-affine. 

Our aim in this paper, then, is to prove that, even when $k$ is countable, the skew fields in various wide classes of noncommutative transcendental division algebras, finitely generated over $k$ as skew fields, cannot be $k$-affine. To our knowledge, the analysis below is the first to provide verified examples of this type.  These examples are not new as skew fields, but their nonaffineness is newly verified. We also address the question of when the skew fields in these classes can be affine over their centers.

To compare the above with the commutative case, recall that the weak Nullstellensatz (Zariski's Lemma) \cite[Statement H$^n_3$]{Zar} ensures that a (commutative) field extension $K$ of $k$ is affine as a $k$-algebra if and only if $K$ is finite dimensional over $k$, if and only if $K$ is affine and algebraic over $k$. In particular, the rational function field $k(x)$, for any choice of field $k$, is not $k$-affine.

We note that questions of affineness over a base field versus affineness over a center are closely related, due to the following: If a $k$-division algebra $D$ is affine over $k$ and finite dimensional over $Z(D)$, then $D$ is finite dimensional over $k$.  Namely, under these conditions the Artin-Tate Lemma (e.g., \cite[Lemma 13.9.10]{McC-Rob}) implies that $Z(D)$ is affine over $k$. The weak Nullstellensatz then ensures that $Z(D)$ is finite dimensional over $k$, and therefore $D$ is finite dimensional over $k$.  Consequently, if $\mathcal D$ is a class of division algebras for which it is known that (affine over the center $\implies$ finite dimensional over the center), then the implication (affine over $k$ $\implies$ finite dimensional over $k$) holds within $\mathcal D$.

\begin{organization}
Our study below is composed of three parts. In \S 2, we show that the quotient division ring of a hereditary noetherian domain cannot be affine if the domain has infinitely many distinct isomorphism classes of simple modules. In \S3 we show that the theoretical results of \S2 apply to suitably conditioned iterated Ore extensions (skew polynomial rings). In \S 4, the work of \S 3 is then applied to specific examples arising in Lie theory and quantum groups.  In particular:

\begin{thmA*} {\rm[Theorem \ref{Dnk.nonaff}]}
Let $D_n(k)$ be the Weyl skew field $\Fract A_n(k)$ for a field $k$ and $n\in\Zpos$.  Then $D_n(k)$ is affine over its center if and only if $D_n(k)$ is finite dimensional over its center, if and only if $\ch k > 0$.  However, $D_n(k)$ is never affine over $k$.
\end{thmA*}

\begin{thmB*} {\rm[Theorems \ref{Q.itOre.nonaff},  \ref{aff.Q.qaffine}]}

Let $R = \Obfq(k^n)$ be a quantum affine space over a field $k$ for $n > 0$.  Then $\Fract R$ is affine over its center if and only if $\Fract R$ is finite dimensional over its center, if and only if all entries of $\bfq$ are roots of unity.  However, $\Fract R$ is never affine over $k$.
\end{thmB*}

Since the 1960s (e.g., \cite{Gel-Kir}), Weyl skew fields have played a significant role in the theory of enveloping algebras of finite dimensional Lie algebras. In particular, the well known Gel'fand-Kirillov conjecture hypothesized that if $\gfrak$ is an algebraic finite-dimensional Lie algebra over a field $k$ of characteristic zero, then $\Fract U(\gfrak) \cong D_n(K)$ for some $n\ge0$ and some rational function field $K$ (in finitely many variables) over $k$.  Although not true in general, this conjecture has been proved for algebraic solvable Lie algebras over algebraically closed fields, for $\mathfrak{sl}_n(k)$, and for $\mathfrak{gl}_n(k)$; see, e.g., \cite{Ale-O-VdB} for a survey. Weyl skew fields also appear in the theory of prime factor algebras of enveloping algebras; see for instance \cite[Chapter 14]{McC-Rob}.

A parallel quantum Gel'fand-Kirillov Conjecture posits that for generic $k$-al\-ge\-bras $A$ among quantized enveloping algebras and quantized coordinate rings, $\Fract A$ is isomorphic to the skew field of fractions of a quantum affine space over a rational function field (in finitely many variables) over $k$.  Moreover, when $A$ is a generic quantized coordinate ring and $P$ is a prime ideal of $A$, $\Fract A/P \cong \Fract \Obfq(K^n)$ for some $n$ and some field extension $K$ of $k$.  These conjectures have been verified widely; see \cite[\S II.10.4]{Bro-Goo} for an overview, and \cite[Section 6]{Cau} for skew fields associated with a broad class of iterated skew polynomial algebras covering many quantized coordinate rings and related algebras.

\end{organization}

\begin{notation} Throughout, $k$ denotes an arbitrary (com\-mu\-ta\-tive) field, not necessarily countable.  All rings in this paper are associative with unit.  A division algebra $D$ over $k$ is \emph{finite over $k$} if $D$ is finite dimensional over $k$. The Goldie quotient ring (semisimple ring of fractions) of a (right or left) semiprime Goldie ring $R$ is denoted $\Fract R$, or $Q(R)$ for short.

We write $R[y;\tau,\delta]$ and $R[y^{\pm 1};\tau]$ for skew polynomial and skew-Laurent polynomial rings, respectively, and $R(y;\tau,\delta)$ and $R(y^{\pm1};\tau)$ for their Goldie quotient rings.  For all skew polynomial rings in this paper, $\tau$ is assumed to be an automorphism.  When $R$ is an algebra over a field $k$, we assume that $\tau$ and $\delta$ are $k$-linear, so that $R[y;\tau,\delta]$, $R[y^{\pm1};\tau]$, $R(y;\tau,\delta)$ and $R(y^{\pm1};\tau)$ are also $k$-algebras.
\end{notation}

\section{Localizing Hereditary Domains}

We make use of the torsion-theoretic localizations of HNP (hereditary noetherian prime) rings developed in \cite{Go74}.  The key instances of HNP rings we require are (noncommutative) PIDs -- principal right and left ideal domains -- along with the first Weyl algebra $A_1(k)$ in characteristic zero.

Recall that if $R$ is an HNP ring and $I$ is an essential right ideal of $R$, then $R/I$ has finite length (e.g., \cite[Proposition 5.4.5]{McC-Rob}).  Given any family $X$ of simple right $R$-modules, let $\calS_X$ denote the collection of those essential right ideals $I$ of $R$ such that $R/I$ has no composition factors isomorphic to any of the modules in $X$, and set
\begin{equation*}
R_X := \{ s \in Q(R) \mid sI \subseteq R\ \text{for some}\ I \in \calS_X \}.
\end{equation*}

We first collect the required tools. 

\begin{theorem} \cite{Go74}
     \label{HNP.things}
Let $R$ be a non-artinian HNP ring and $\calW$ a set containing exactly one member of each isomorphism class of simple right $R$-modules.

{\rm(a)} Every ring $T$ between $R$ and $Q(R)$ is HNP, with $Q(T) = Q(R)$.

{\rm(b)} The assignment $X \mapsto R_X$ gives an order-reversing bijection between the set of subsets of $\calW$ and the set of rings between $R$ and $Q(R)$.

{\rm(c)} For $X \subseteq \calW$, the simple right $R_X$-modules are exactly those modules isomorphic to $A \otimes_R R_X$ for $A \in X$.
\end{theorem}

\begin{proof}
\cite[Proposition 2 and Theorems 3, 5]{Go74}
\end{proof}

\begin{theorem}  \label{inf.simples>inf.gen}
Let $R$ be an HNP ring which has infinitely many isomorphism classes of simple right modules.  Then $Q(R)$ cannot be generated as a ring by $R$ together with finitely many elements of $Q(R)$.  In particular, $Q(R)$ is not affine as a $Z(R)$-algebra.
\end{theorem}

\begin{proof}
Let $F \subseteq Q(R)$ be finite, $T$ the subring of $Q(R)$ generated by $R \cup F$, and $\calW$ as in Theorem \ref{HNP.things}.  For $f \in F$, the $R$-module $(R+fR)/R$ has finite length because it is isomorphic to $R/I$ where $I := \{ r \in R \mid fr \in R \}$ is an essential right ideal of $R$.  Thus, the collection $Y$ of those modules in $\calW$ which appear as composition factors in any of the $R$-modules $(R+fR)/R$ for $f\in F$ is finite.

Now let $X := \calW \setminus Y$.  Then $F \subseteq R_X$ and so $T \subseteq R_X$.  Since $\calW$ is infinite (by hypothesis) and $Y$ is finite, $X$ must be infinite.  Choose a proper subset $Z \subset X$.  By Theorem \ref{HNP.things}, $R_X \subsetneq R_Z \subseteq Q(R)$, and therefore $T \ne Q(R)$.
\end{proof}

\begin{cor}  \label{main.nonaff}
Let $R$ be an HNP ring which is an algebra over a field $k$.  If $R$ has infinitely many isomorphism classes of simple right modules, then $Q(R)$ is not affine as a $k$-algebra.
\qed\end{cor}

We give two immediate applications. 

\begin{cor}  \label{D(z).nonaff} 
Let $D[z]$ be a polynomial ring over a division ring $D$ and $k := Z(D)$.  Then $D(z)$ is not affine over $k$.
\end{cor}

\begin{proof}
First, recall that the center of $R := D[z]$ is $k[z]$ and that every two-sided ideal of $R$ is generated by an element of $Z(R)$; see, e.g., \cite[Proposition 17.1]{Goo-War}. Next, let $f_1,f_2, \dots$ be infinitely many pairwise nonassociate, irreducible polynomials in $k[z]$. We see that the ideals $Rf_i$ must be pairwise distinct maximal ideals of $R$.  Choose a simple right $(R/Rf_i)$-module $A_i$ for each $i$; then $A_1,A_2,\dots$ are pairwise non-isomorphic simple right $R$-modules.  It therefore follows from Corollary \ref{main.nonaff} that $D(z) = Q(R)$ is not affine over $k$.
\end{proof}

A second application concerns the first Weyl skew field, $D_1(k) = \Fract A_1(k)$, over a field $k$. We deal with $D_n(k)$ for $n>1$ later (Theorem \ref{Dnk.nonaff}).

\begin{exam}  \label{D1k.nonaff}
If $k$ is a field then $D_1(k)$ is not affine over $k$.
\end{exam} 

\begin{proof} 
Assume first that $k$ has characteristic zero. Then the domain $A_1(k)$ is HNP, and by \cite[Proposition 5.6]{McR73}, $A_1(k)$ has infinitely many isomorphism classes of simple right modules. Thus the desired conclusion follows from Corollary \ref{main.nonaff}.

Now assume that $k$ has characteristic $p>0$.  Then $A_1(k)$ is not HNP.  Instead, view $D_1(k) = Q(R)$, where $R := k(x)[y;d/dx]$, a PID (and so HNP). Then $R$ is a finite rank free module over the central subalgebra $S := k(x^p)[y^p]$, a polynomial ring. Choose an infinite sequence of pairwise nonassociate, irreducible polynomials $f_1,f_2,\dots$ in $S$. Since the $f_i$ generate pairwise distinct maximal ideals of $S$, the ideals $Rf_i$ of $R$ are pairwise comaximal in $R$. Similarly to Corollary \ref{D(z).nonaff}, we can choose a simple right $(R/Rf_i)$-module for each $i$, yielding an infinite sequence of pairwise non-isomorphic simple right $R$-modules. The desired conclusion now follows from Corollary \ref{main.nonaff}.
\end{proof}

The converse of Corollary \ref{main.nonaff} does not hold in general, as follows.

\begin{exam}  \label{CozOso.nonaff}
Let $p$ be a prime, $K$ an algebraic closure of $k := \FF_p$, and $\tau$ the Frobenius automorphism of $K$.  The $k$-algebra $R := K[y^{\pm1};\tau]$ is a PID which has only one isomorphism class of simple right modules, but $Q(R)$ is not affine over $k$, which is its center.
\end{exam}

\begin{proof}
Cozzens \cite[Theorem 2.3]{Coz} and Osofsky \cite[Proposition 4]{Oso} showed that $R$ has only one isomorphism class of simple right modules. That $k = Z(Q(R))$ follows from, e.g., \cite[Exercise 6\,I]{Goo-War}. 

Let $\calF$ be the collection of finite subfields of $K$, an upward directed family whose union is $K$.  Each $F \in \calF$ is perfect, hence invariant under $\tau$, and so $R$ contains $T_F := F[y^{\pm1};\tau|_F]$ as a $k$-sub\-algebra.  These subalgebras form an upward directed family whose union is $R$, and $(Q(T_F))_{F\in\calF}$ is an upward directed family whose union is $Q(R)$.

If $Q(R)$ is affine over $k$, then $Q(R) = Q(T_L)$ for some $L \in \calF$.  Let $p^n := \card(L)$; then $\tau^n|_L = \Id_L$ and so $y^n$ centralizes $L$.  But then $y^n$ is central in $Q(T_L)$, contradicting the fact that $Z(Q(R)) = k$.
\end{proof}

\begin{rem} Again consider the algebra $R := K[y^{\pm 1};\tau]$ from Example \ref{CozOso.nonaff}, and choose any element $q \in Q(R) \setminus R$. It follows from Theorem \ref{HNP.things} that $Q(R)$ can be generated as a $k$-algebra by $R \cup \{q\}$. In particular, while $Q(R)$ is not affine over a central subfield, it does hold that $Q(R)$ can be generated as a ring by $q, y, y^{-1}$ and the (noncentral) subfield $K$. 
\end{rem}

In contrast to the above non-affineness results, we present one case in which affineness does extend. Recall that an algebra $R$ over a field $k$ is \emph{matrix algebraic} over $k$ provided all of the matrix algebras $M_n(R)$ are algebraic over $k$. The question posed by Jacobson in \cite[p.695]{Jac-1945} of whether or not algebraicity implies matrix algebraicity is still open in general; Amitsur gave a positive solution over uncountable fields in \cite{Ami}.

\begin{prop}
Let $D[z]$ be a polynomial ring over a division ring $D$.  If $D$ is affine and matrix-algebraic over its center $K$, then $D(z)$ is affine over its center $K(z)$.
\end{prop}

\begin{proof}
Set $E := K(z) \otimes_K D$, and note that $Z(E) = K(z) \otimes 1$. We know that $E$ is a simple ring, since $K(z)$ is simple and $D$ is central simple over $K$, and so the natural map $\eta : E \rightarrow D(z)$ must be injective. Therefore, $E$ is a domain. Also, $E$ and $\eta(E)$ are clearly $K(z)$-affine. Next, since $D$ is matrix-algebraic over $K$, it follows from \cite[Theorem 17.2]{Goo-War} that $E$ is a division algebra. But $\eta(E)$ is generated as a division algebra over $K$ by $D$ and $z$, and so $\eta(E)$ must equal $D(z)$. It follows that $Z(D(z)) = \eta(Z(E)) = K(z)$ and that $D(z) = \eta(E)$ is affine over $K(z)$.
\end{proof}

\begin{rem}
The results in this section can also be compared with the following purely polynomial-identity-theoretic observation: If $D[z]$ is a polynomial ring over a division ring $D$, then $D(z)$ is finite over its center if and only if $D$ is finite over its center.
(Proof: By Kaplansky's Theorem, see, e.g., \cite[Theorem 13.3.8]{McC-Rob}, a division algebra $D$ is PI if and only if it is finite over its center.  Also, $D(z)$ is PI if and only if $D[z]$ is PI, if and only if $D$ is PI.)
\end{rem}

\section{PIDs with infinitely many simple modules}

In order to apply Theorem \ref{inf.simples>inf.gen} and/or Corollary \ref{main.nonaff} to skew fields of the form $D(y;\tau,\delta)$, where $D$ is a division algebra, we would like to have infinite families of simple modules over  the PIDs $D[y;\tau,\delta]$.  This can be obtained from information about eigenvalues of $\tau$ or $\delta$, as follows.

\begin{lem}  \label{iso.simples.vs.eigenvals}
Let $R = D[y;\tau,\delta]$ where $D$ is a division algebra over a field $k$.

{\rm(a)} Let $a,b \in D$.  Then $R/(y-a)R \cong R/(y-b)R$ if and only if there is some nonzero $v\in D$ such that $av-\tau(v)b = \delta(v)$.

{\rm(b)} Assume that $\delta = 0$, and let $\alpha,\beta \in \kx$.  Then $R/(y-\alpha)R \cong R/(y-\beta)R$ if and only if $\alpha\beta^{-1}$ is an eigenvalue for $\tau$.

{\rm(c)} Assume that $\tau = \Id_D$, and let $\alpha,\beta \in k$. Then $R/(y-\alpha)R \cong R/(y-\beta)R$ if and only if $\alpha - \beta$ is an eigenvalue for $\delta$.
\end{lem}

\begin{proof}
For any $c\in D$, the right $R$-module $R/(y-c)R$ is simple, being $1$-di\-men\-sional over $D$ on the right.

(a) If $R/(y-a)R \cong R/(y-b)R$, there is a nonzero element $x \in R/(y-a)R$ such that $x(y-b) = 0$.  We may write $x = u+(y-a)R$ for some nonzero $u \in D$, and then $u(y-b) = (y-a)v$ for some nonzero $v \in R$.  Since $u(y-b)$ has degree $1$, $v$ lies in $D$.  Comparing coefficients, we find that $u = \tau(v)$ and $- ub = \delta(v) - av$, whence $av-\tau(v)b = \delta(v)$.  The converse is obtained by reversing the above steps.

(b) In view of part (a), the isomorphism occurs if and only if there is some nonzero $v \in D$ such that $\alpha v = \tau(v) \beta$, that is, $\tau(v) = \alpha \beta^{-1} v$.

(c) In this case, the isomorphism occurs if and only if there is some nonzero $v \in D$ such that $\delta(v) = (\alpha-\beta)v$.
\end{proof}

\begin{cor}  \label{inf.noneigenvals>inf.simples}
Let $R = D[y;\tau,\delta]$ where $D$ is a division algebra over a field $k$.

{\rm(a)} Assume that $\delta = 0$, and let $E_\tau$ be the subgroup of $\kx$ consisting of the $\tau$-eigen\-values of the $\tau$-eigenvectors in $D$.  If $\kx/E_\tau$ is infinite, then $R$ has infinitely many isomorphism classes of simple right modules.

{\rm(b)} Assume that $\tau = \Id_D$, and let $E_\delta$ be the subgroup of $(k,+)$ consisting of the $\delta$-eigen\-values of the $\delta$-eigenvectors in $D$.  If $(k,+)/E_\delta$ is infinite, then $R$ has infinitely many isomorphism classes of simple right modules.
\end{cor}

\begin{proof}
(a) If $\alpha,\beta \in \kx$ lie in different cosets of $E_\tau$ in $\kx$, then by Lemma \ref{iso.simples.vs.eigenvals}(b), the simple right $R$-modules $R/(y-\alpha)R$ and $R/(y-\beta)R$ are not isomorphic.

(b) This follows similarly from Lemma \ref{iso.simples.vs.eigenvals}(c).
\end{proof}

The following facts about additive and multiplicative groups of fields are seemingly well known, but we have found no references beyond discussion forums and blogs, so we include a short proof.

\begin{lem}  \label{kxk+.non.fingen}
If $k$ is an infinite field, then neither of the groups $\kx$ nor $(k,+)$ is finitely generated.
\qed\end{lem}

\begin{proof}
For the contrapositive, assume that $\kx$ or $(k,+)$ is finitely generated.

First suppose that $\ch k = 0$.  Since subgroups of finitely generated abelian groups are finitely generated, either $\QQ^\times$ or $(\QQ,+)$ must be finitely generated.  Consequently, $\QQ$ is finitely generated as a ring, say by $q_1,\dots,q_t$.  Write $q_i = a_i/b$ for some $a_i \in \ZZ$, $b \in \Zpos$, and choose a prime $p$ not dividing $b$.  Then $1/p = a/b^n$ for some $a \in \ZZ$, $n \in \Zpos$ and so $p \mid b^n$, which is impossible.

Thus $\ch k = p > 0$.  Our assumption implies that $k$ is finitely generated as an algebra over $\FF_p$.  The weak Nullstellensatz then says that $\dim_{\,\FF_p}(k)$ is finite, and therefore $k$ is finite.
\end{proof}

We'll also find the following result of Bavula to be useful: 

\begin{theorem}  \label{Bav.fin.gen.eigen}  \cite{Bav}
Let $D$ be a division algebra over a field $k$, and assume that the algebra $D\otimes_k D^{\op}$ is noetherian.

{\rm(a)} If $\tau$ is a $k$-algebra automorphism of $D$, then the set of eigenvalues of $\tau$ is a finitely generated subgroup of $\kx$.

{\rm(b)} If $\delta$ is a $k$-linear derivation of $D$, then the set of eigenvalues of $\delta$ is a finitely generated subgroup of $(k,+)$.
\end{theorem}

\begin{proof}
\cite[Theorem 1.1]{Bav}
\end{proof}

\begin{rem}  \label{some.DtensorDop.noeth}
The class of division algebras $D$ over a field $k$ for which $D\otimes_k D^{\op}$ is noetherian contains the skew fields of fractions of the following algebras:  iterated skew polynomial rings over $k$ (e.g., \cite[Theorem 17.25]{Goo-War}), enveloping algebras of finite dimensional Lie algebras over $k$ \cite[Corollary 2.5]{Bav}, rings of differential operators on smooth irreducible affine varieties over $k$ when $\ch k = 0$ \cite[Corollary 2.4]{Bav}.  This class also contains the \emph{stratiform} division algebras introduced by Schofield \cite{Scho} (c.f. \cite[Proposition 1.8]{Yek-Zha}).
\end{rem}

\begin{theorem}  \label{PIDs.w.Q.nonaff}
Let $R$ be either $D[y;\tau]$ or $D[y;\delta]$, where $D$ is a division algebra over an infinite field $k$.  If $D\otimes_k D^{\op}$ is noetherian, then $R$ has infinitely many isomorphism classes of simple right modules.  Consequently, $Q(R)$ is not affine over $k$.
\end{theorem}

\begin{proof}
The groups $E_\tau$ and $E_\delta$ of Corollary \ref{inf.noneigenvals>inf.simples} are finitely generated by Theorem \ref{Bav.fin.gen.eigen}.  In view of Lemma \ref{kxk+.non.fingen}, the quotients $\kx/E_\tau$ and $(k,+)/E_\delta$ are infinite, and therefore the theorem follows from Corollaries \ref{inf.noneigenvals>inf.simples} and \ref{main.nonaff}.
\end{proof}

\section{Applications}

\begin{theorem}  \label{PIDs.w.Q.nonaff/Z}
Let $D$ be a division algebra over an infinite field $k$, and assume that $D\otimes_k D^{\op}$ is noetherian.

{\rm(a)} Let $R = D[y;\delta]$.  If $\ch k = 0$ and the derivation $\delta$ is not inner, then $Q(R)$ is not affine over its center.

{\rm(b)} Let $R = D[y;\tau]$.  If no nonzero power of the automorphism $\tau$ is inner, then $Q(R)$ is not affine over its center.
\end{theorem}

\begin{proof}
(a) Set $K := Z(D)\cap \ker \delta$; then $K$ is a field, $D$ is a $K$-algebra and $\delta$ is $K$-linear.  Since $D \otimes_K D^{\op}$ is a homomorphic image of $D\otimes_k D^{\op}$, the algebra $D \otimes_K D^{\op}$ is noetherian.  Thus, Theorem \ref{PIDs.w.Q.nonaff} implies that $Q(R)$ is not affine over $K$.   It remains to show that $K = Z(Q(R))$.

Under the given hypotheses, $R$ is a simple ring (e.g., \cite[Proposition 2.1]{Goo-War}).  If $z \in Z(Q(R))$, then $I := \{ r \in R \mid zr \in R \}$ is a nonzero ideal of $R$, whence $I = R$ and so $z \in R$.  Thus, $Z(Q(R)) = Z(R)$.  In particular, $Z(R)$ is a field.  Since the invertible elements of $R$ are the nonzero elements of $D$, we obtain $Z(R) \subseteq D$, from which it is clear that $Z(R) = K$.

(b) Let $K$ be the fixed field $Z(D)^\tau$.  As in part (a), Theorem \ref{PIDs.w.Q.nonaff} implies that $Q(R)$ is not affine over $K$, so we just need to show that $Z(Q(R)) = K$.

Since the skew-Laurent algebra $T := D[y^{\pm1};\tau]$ is simple (e.g., \cite[Theorem 1.17]{Goo-War}), we find as above that $Z(Q(R)) = Z(Q(T)) = Z(T)$.  The invertible elements of $T$ have the form $t = uy^n$ for $u \in D\setminus \{0\}$ and $n \in \ZZ$.  Such an element $t$ commutes with $y$ if and only if $\tau(u) = u$, and $t$ commutes with the elements of $D$ if and only if $\tau^n(d) = u^{-1}du$ for all $d\in D$.  Thus $Z(T) = K$, as desired.
\end{proof}

\begin{theorem}  \label{Q.itOre.nonaff}
Let $R = k[x_1][x_2;\tau_2,\delta_2]\cdots [x_n;\tau_n, \delta_n]$ be an iterated skew polynomial algebra over an infinite field $k$, where $n\ge1$ and either $\delta_n = 0$ or $\tau_n = \Id$.  Then $Q(R)$ is not affine over $k$.
\end{theorem}

\begin{proof}
View $R = S[x_n;\tau_n, \delta_n]$ and $Q(R) = D(x_n;\tau_n, \delta_n)$ where
\begin{equation*}
S := k[x_1][x_2;\tau_2,\delta_2]\cdots [x_{n-1};\tau_{n-1}, \delta_{n-1}]
\end{equation*}
and $D := Q(S)$.  Since $D \otimes_k D^{\op}$ is noetherian (Remark \ref{some.DtensorDop.noeth}), the result follows from Theorem \ref{PIDs.w.Q.nonaff}.
\end{proof}

While Theorem \ref{Q.itOre.nonaff} immediately implies that the skew field of fractions of any quantum affine space $\Obfq(k^n)$ is not affine over $k$, that does not address the question of affineness over centers.  We deal with this shortly 

\begin{theorem}  \label{Dnk.nonaff}
Let $k$ be a field and $n \in \Zpos$.

{\rm(a)} If $\ch k = 0$, then $D_n(k)$ is not affine over its center, which is $k$.

{\rm(b)} If $\ch k > 0$, then $D_n(k)$ is finite over its center, but is not affine over $k$.
\end{theorem}

\begin{proof}
(a) It is well known that $Z(D_n(k)) = k$ in characteristic zero (e.g., \cite[Exercise 17\,O]{Goo-War}).  The result thus follows from Theorem \ref{Q.itOre.nonaff}.

(b) In this case $A_n(k)$ is a finite rank free module over the central subalgebra $k[x_1^p,y_1^p,\dots,x_n^p,y_n^p]$, where $p := \ch k$, and the first conclusion follows.  The second conclusion follows from Theorem \ref{Q.itOre.nonaff} if $k$ is infinite, but it also holds when $k$ is finite, as follows.  View $D_n(k) = Q(R)$ where $R := D_{n-1}(k)(x_n)[y_n; d/dx_n]$.  Then $R$ is a free module over the central subalgebra $S := k[y_n^p]$.  As in Example \ref{D1k.nonaff}, $R$ has an infinite sequence of pairwise non-isomorphic simple right modules, and we obtain the conclusion from Corollary \ref{main.nonaff}.
\end{proof}

\begin{theorem}  \label{Q.Ug.Ore.nonaff}
Let $R$ be $U(\gfrak)[y;\tau]$ or $U(\gfrak)[y;\delta]$ where $\gfrak$ is a finite dimensional Lie algebra over an infinite field $k$.  Then $Q(R)$ is not affine over $k$.
\end{theorem}

\begin{proof}
Remark \ref{some.DtensorDop.noeth} and Theorem \ref{PIDs.w.Q.nonaff}.
\end{proof}

\begin{cor}  \label{Q.Ug.nonaff}
Let $\gfrak$ be a finite dimensional Lie algebra over an infinite field $k$.  If $\gfrak$ contains an ideal $\hfrak$ of codimension $1$, then $Q(U(\gfrak))$ is not affine over $k$.
\end{cor}

\begin{proof}
Here $U(\gfrak) = U(\hfrak)[y;\delta]$ for suitable $y$ and $\delta$. 
\end{proof}

Corollary \ref{Q.Ug.nonaff} applies in particular when $\gfrak$ is completely solvable, but that situation is also covered directly by Theorem \ref{Q.itOre.nonaff}.

We now turn to quantum affine spaces $\Obfq(k^n)$ and quantum tori $\Obfq((k^\times)^n)$.  In this notation, $n \in \Zpos$ and $\bfq$ is a multiplicatively antisymmetric $n\times n$ matrix of nonzero scalars from $k$. 

While the following lemma may be well known, we are unaware of a suitable reference.

\begin{lem}  \label{ZQ=QZ.qtorus}
If $T = \Obfq((k^\times)^n)$ is a quantum torus over a field $k$, then $Z(Q(T))\linebreak = Q(Z(T))$.
\end{lem}

\begin{proof}
Give $T$ its natural $\ZZ^n$-grading, and observe that $T$ is graded-simple.  Consequently, all ideals of $T$ are centrally generated \cite[Proposition II.3.8]{Bro-Goo}.

Given $z \in Z(Q(T))$, the set $I := \{ t \in T \mid zt \in T \}$ is a nonzero ideal of $T$.  Choose a nonzero element $u \in I \cap Z(T)$.  Then $zu \in Z(T)$ and so $z = (zu)u^{-1} \in Q(Z(T))$.  Therefore $Z(Q(T)) \subseteq Q(Z(T))$.  The reverse inclusion is clear.
\end{proof}

\begin{theorem}  \label{aff.Q.qaffine} 
Let $R = \Obfq(k^n)$ be a quantum affine space over a field $k$.

{\rm(a)} If all the entries $q_{ij}$ of $\bfq$ are roots of unity, then $Q(R)$ is finite over its center.

{\rm(b)} If at least one $q_{ij}$ is not a root of unity, then $Q(R)$ is not affine over its center.
\end{theorem}

\begin{proof}
Present $R =  k\langle x_1,\dots,x_n \mid x_ix_j = q_{ij} x_jx_i \; \forall\; i,j \in [1,n] \rangle$.  Let $T := \Obfq((k^\times)^n)$; then $Q(R) = Q(T)$.
Set $\Gamma := \ZZ^n$ and $x^\gamma := x_1^{\gamma_1} \cdots x_n^{\gamma_n}$ for $\gamma = (\gamma_1,\dots,\gamma_n) \in \Gamma$; then $T = \bigoplus_{\gamma\in\Gamma} k x^\gamma$ and $x^\alpha x^\beta \in k^\times x^{\alpha+\beta}$ for all $\alpha,\beta \in \Gamma$.  Now $Z(T) = \bigoplus_{\gamma\in\Gamma_Z} k x^\gamma$ where $\Gamma_Z := \{ \gamma \in \Gamma \mid x^\gamma \in Z(T) \}$ is a subgroup of $\Gamma$. 

(a)  Assuming all $q_{ij}$ are roots of unity, there is a positive integer $r$ such that $q_{ij}^r = 1$ for all $i$, $j$, whence all $x_i^r$ are central in $R$.  Now $R$ is module-finite over the central subalgebra $Z' := k\langle x_1^r,\dots,x_n^r \rangle$, and so $Q(R)$ is finite over the central subfield $Q(Z')$.  Therefore $Q(R)$ is finite over its center.

(b) Assume that $q_{lm}$ is not a root of unity for some $l$, $m$.  In particular, this forces $k$ to be infinite.  Now $x_l^r$ fails to commute with $x_m$ for any nonzero integer $r$, and so if $e_l$ is the $l$-th standard basis vector in $\Gamma$, the coset $e_l + \Gamma_Z$ has infinite order in $\Gamma/\Gamma_Z$.  Thus, $\Gamma/\Gamma_Z$ is infinite.

Since $\Gamma$ is finitely generated, it follows that $\Gamma$ has a subgroup $\Gamma' \supseteq \Gamma_Z$ such that $\Gamma/\Gamma'$ is infinite cyclic, say with a generator $\beta+\Gamma'$.  Now $S := \bigoplus_{\gamma\in\Gamma'} k x^\gamma$ is a subalgebra of $T$, isomorphic to a quantum torus of rank $n-1$, and $T$ is a free left $S$-module with basis $(x^{n\beta} \mid n\in \ZZ)$.  Set $y := x^{\beta}$.  The inner automorphism $y(-)y^{-1}$ restricts to an automorphism $\tau$ of $S$, and $T = S[y^{\pm1};\tau]$.  Thus $Q(R) = Q(T) = Q(S[y;\tau]) = Q(E[y;\tau])$ where $E := Q(S)$.  By construction, $Z(T) \subseteq S \subseteq E$, and consequently $Z(Q(R)) \subseteq E$ in view of Lemma \ref{ZQ=QZ.qtorus}.

Since $E = Q(S')$ where $S'$ is the quantum affine space (over $k$) corresponding to $S$, the algebra $E \otimes_k E^{\op}$ is noetherian (Remark \ref{some.DtensorDop.noeth}).  

If some nonzero power of $\tau$ is inner on $E$, then  by \cite[Theorems 0,1]{Jor}, there exist a nonzero element $u \in E$ and a positive integer $n$ such that $\tau(u) = u$ and $\tau^n = u^{-1}(-)u$.  It follows that $uy^n \in Z(E[y^{\pm1};\tau]) \subseteq Z(Q(R))$, which is impossible because $Z(Q(R)) \subseteq E$. Therefore no nonzero power of $\tau$ is inner on $E$.

Theorem \ref{PIDs.w.Q.nonaff/Z} now implies that $Q(R)$ is not affine over its center.
\end{proof}

We conclude with the following general question: Given a $k$-division algebra $D$, equipped with a $k$-linear automorphism $\tau$ and $\tau$-derivation $\delta$, must $D(y;\tau,\delta)$ be nonaffine over $k$? 

\section*{Acknowledgement}  The authors thank the workshop \emph{Revisiting Fundamental Problems: Infinite Dimensional Division Algebras -- Algebraicity and Finiteness} at the Simons Laufer Mathematical Sciences Institute in December 2025, which motivated this work and where our collaboration on it began. They are grateful to the referee for thoughtful comments and suggestions.


\bibliographystyle{amsplain}

\end{document}